\documentclass[runningheads]{llncs}

\newcommand{\refsec}[1] {Section~\ref{#1}}
\newcommand{\refeq}[1] {(\ref{#1})}

\newcommand{\reffig}[1] {Figure~\ref{#1}}

\newcommand{\refapx}[1] {\ref{#1}}

\renewcommand{\vec}[1] {\boldsymbol{#1}}
\newcommand{\mat}[1] {\boldsymbol{\mathsf{#1}}}

\usepackage[utf8x]{inputenc}

\usepackage{amsmath}
\usepackage{amssymb}
\usepackage{graphicx}
\usepackage{ucs}
\usepackage[utf8x]{inputenc}

\usepackage{subcaption}
\usepackage{tikz}
\usepackage[outline]{contour}
\usepackage{enumerate}

\usepackage[T1]{fontenc}
\usepackage{amsmath,amssymb}
\usepackage{authblk}
\usepackage{booktabs}
\usepackage{hyperref}
\hypersetup{
    colorlinks,
    hidelinks,
    citecolor=blue,
    filecolor=black,
    linkcolor=red,
    urlcolor=green,
    plainpages=false,
    pdfpagelabels=true
}

\usepackage{url}
\newcommand{\keywords}[1]{\par\addvspace\baselineskip
\noindent\keywordname\enspace\ignorespaces#1}

\newcommand{\R}{\mathbb{R}}
\newcommand{\E}{\mathbb{E}}
\newcommand{\be}{\vec{e}}
\newcommand{\br}{\vec{r}}
\newcommand{\bd}{\vec{d}}
\newcommand{\bk}{\vec{\kappa}}
\newcommand{\bm}{\vec{m}}
\newcommand{\bn}{\vec{n}}

\newcommand{\bw}{\vec{\omega}}
\newcommand{\bv}{\vec{v}}
\newcommand{\bnu}{\vec{\nu}}
\newcommand{\bJ}{\vec{J}}
\newcommand{\bl}{\vec{L}}
\newcommand{\bh}{\vec{h}}
\newcommand{\bF}{\vec{F}}
\newcommand{\bp}{\vec{p}}
\newcommand{\bq}{\vec{q}}

\begin{document}

\mainmatter

\title{Symbolic-Numeric Integration of the\\Dynamical Cosserat Equations}
\titlerunning{Symbolic-Numeric Integration of the Dynamical Cosserat Equations}

\author{Dmitry~A.~Lyakhov\inst{1}
\and Vladimir~P.~Gerdt\inst{3,4} \and Andreas~G.~Weber\inst{5} \and Dominik~L.~Michels\inst{1,2} }
\institute{
Visual Computing Center, King Abdullah University of Science and Technology, Al Khawarizmi Building, Thuwal, 23955-6900, Kingdom of Saudi Arabia\\
\email{\{dmitry.lyakhov,dominik.michels\}@kaust.edu.sa}\\
\and
Department of Computer Science, Stanford University, 353 Serra Mall, Stanford, CA 94305, United States of America\\
\email{michels@cs.stanford.edu}\\
\and
Laboratory of Information Technologies, Joint Institute for Nuclear Research, 6 Joliot--Curie St, Dubna, 141980, Russian Federation\\
\and
Peoples' Friendship University of Russia, 6 Miklukho--Maklaya St, Moscow, 117198, Russian Federation\\
\email{gerdt@jinr.ru}\\
\and
Institute of Computer Science II, University of Bonn, Friedrich-Ebert-Allee 144, 53113 Bonn, Germany\\
\email{weber@cs.uni-bonn.de}}

\authorrunning{D.~A.~Lyakhov et al.}
\toctitle{Lecture Notes in Computer Science}
\tocauthor{D.~A.~Lyakhov et al.}


\maketitle

\begin{abstract}
We devise a symbolic-numeric approach to the integration of the dynamical part of the Cosserat equations, a system of nonlinear partial differential equations describing the mechanical behavior of slender structures, like fibers and rods. This is based on our previous results on the construction of a closed form general solution to the kinematic part of the Cosserat system. Our approach combines methods of numerical exponential integration and symbolic integration of the intermediate system of nonlinear ordinary differential equations describing the dynamics of one of the arbitrary vector-functions in the general solution of the kinematic part in terms of the module of the twist vector-function. We present an experimental comparison with the well-established generalized $\alpha$-method illustrating the computational efficiency of our approach for problems in structural mechanics.
\end{abstract}

\section{Introduction}
\label{sec:1}

Deformable-body dynamics can be considered as a subarea of continuous mechanics that studies motion of deformable solids subject to the action of internal and external forces (cf.~\cite{Luo:10}). The equations describing the dynamics of such solids are nonlinear partial differential equations (PDEs) whose independent variables are three spatial coordinates and time. Given a particular deformable mechanical structure, to describe its dynamics, it is necessary to satisfy these equations at each point of the structure together with appropriate boundary conditions. For a mechanical structure having special geometric properties, it is worthwhile to exploit these properties to develop a simplified but geometrically exact mechanical model of the structure. Classical examples of such models are Cosserat theories of shells and rods; see~e.g.~\cite{Rubin:00} and references therein.

Rods are nearly one-dimensional structures whose dynamics can be described by the Cosserat theory of (elastic) rods (cf.~\cite{Antman:1995}, Ch.~8;~\cite{Rubin:00}, Ch.~5; and the original work \cite{Cosserat:1909}). This is a general and geometrically exact dynamical model that takes bending, extension, shear, and torsion into account, as well as rod deformations under external forces and torques. In this context, the dynamics of a rod is described by a governing system of twelve first-order nonlinear partial differential equations (PDEs) with a pair of independent variables $(s,t)$, where $s$ is the arc-length and $t$ the time parameter. In this PDE system, the two kinematic vector equations ((9a)--(9b) in~\cite{Antman:1995}, Ch.~8) are parameter free and represent the compatibility conditions for four vector functions $\bk,\bw,\bnu,\bv$ in $(s,t)$. Whereas the first vector equation only contains two vector functions $\bk$ and $\bw$, the second one contains all four vector functions $\bk,\bw,\bnu,\bv$. The remaining two vector equations in the governing system are dynamical equations of motion and include two more dependent vector variables $\hat{\bm}(s,t)$ and $\hat{\bn}(s,t)$. Moreover, these dynamical equations contain parameters (or parametric functions of $s$) to characterize the rod and to include the external forces and torques. Studying the dynamics of Cosserat rods has various scientific and industrial applications, for example in civil and mechanical engineering (cf.~\cite{Boyer:2011}), microelectronics and robotics (cf.~\cite{CaoTucker:2008}), biophysics (cf.~\cite{Hilfinger:2006} and references therein), and visual computing (cf.~\cite{MMS:2015}).

Because of its inherent stiffness caused by different deformation modes, the treatment of the underlying equations usually requires the application of specific solvers; see e.g.~\cite{MSW:2014}. In order to reduce the computational overhead caused by the stiffness, we employed Lie symmetry based integration methods (cf.~\cite{Ibragimov:09,Olver:93}) and the theory of completion to involution (cf.~\cite{Seiler:10}) to the two kinematic vector equations (cf.~\cite{Antman:1995}, Ch.~8, Eq.~(9a)--(9b)) and constructed their general and analytical solution in \cite{MLGSW:2014,MLGHRKW:2016}, which depends on two arbitrary vector functions in $(s,t)$.

In this contribution, we exploit the general analytic solution to the kinematic part of the governing Cosserat system constructed in \cite{MLGSW:2014,MLGHRKW:2016} and develop a symbolic-numeric  approach to the integration of the dynamical part of the system. Our approach combines the ideas of numerical exponential integration (see e.g.~\cite{Hochbruck:2010,MD:2015} and references therein) and symbolic integration of the intermediate system of nonlinear ordinary differential equations describing the dynamics of the arbitrary vector function in the general solution to the kinematic part in terms of the module of the twist vector function. The symbolic part of the integration is performed by means of Maple. We present an experimental comparison of the computational efficiency of our approach with that of the generalized $\alpha$-method for the numerical integration of problems in structural mechanics (see e.g.~\cite{SobottkaLayWeber:2008}).

\noindent
This paper is organized as follows. In \refsec{sec:2}, we present the governing PDE system in the special Cosserat theory of rods and analytical solution to its kinematic part constructed  in~\cite{MLGSW:2014,MLGHRKW:2016}. In \refsec{sec:3}, we show first that the (naive) straightforward numerical integration of the dynamical part of Cosserat system has a severe obstacle caused by a singularity in the system. Then we describe a symbolic-numeric method to integrate the dynamical equations based on the ideas of exponential integration and the construction of a closed form analytical solution to the underlying nonlinear dynamical system. In doing so, we show that this symbolic-numeric method is free of the singularity problem. In \refsec{sec:4}, we present an experimental comparison of our method with the generalized $\alpha$-method. Some concluding remarks are given in \refsec{sec:5}.

\section{Governing Cosserat Equations and the General Solution of their Kinematic Part}
\label{sec:2}

The governing PDE system in the special Cosserat theory of rods (cf.~\cite{Antman:1995,CaoTucker:2008,Cosserat:1909,MLGSW:2014,MLGHRKW:2016}) can be written in the following form:
\begin{subequations}
 \begin{align}
    & \bk_t=\bw_s - \bw\times \bk\,,\label{kinematic1}\\[0.1cm]
    & \bnu_t=\bv_s+\bk\times\bv-\bw\times \bnu\,,\label{kinematic2}\\[0.15cm]
    & \rho\bJ\cdot \bw_t=\hat{\bm}_s+\bk\times\hat{\bm}+\bnu\times
        \hat{\bn}-\bw\times(\rho\bJ\cdot\bw)+\bl\,,\label{dynamic1}\\[0.15cm]
    & \rho A\bv_t=\hat{\bn}_s+\bk\times \hat{\bn}-\bw\times (\rho A\bv)+\bF\,.\label{dynamic2}
 \end{align}
\end{subequations}
Here, the independent variable $t$ denotes the time and another independent variable $s$ the arc-length parameter identifying a {\em material cross-section} of the rod, which consists of all material points whose reference positions are on the plane perpendicular to the rod at $s$. The
{\em Darboux} vector-function $\bk=\sum_{k=1}^3\kappa_k\bd_k$ and the {\em twist} vector-function $\bw=\sum_{k=1}^3\omega_k\bd_k$ are determined by the kinematic relations
\begin{equation*}
         \partial_s\bd_k=\bk\times \bd_k\,,\quad \partial_t\bd_k=\bw\times \bd_k\,,\label{kr}
\end{equation*}
where the vectors $\bd_1,\,\bd_2$, and $\bd_3:=\bd_1\times \bd_2$ form a right-handed orthonormal moving frame. These vectors are called {\em directors}. The use of the triple $(\bd_1,\,\bd_2,\,\bd_3)$ is natural for the intrinsic description of the rod deformation. Moreover, $\br$ describes the motion of the rod relative to the fixed frame $(\be_1,\,\be_2,\,\be_3)$. This is illustrated in \reffig{Fig1}.

In doing so, the motion of a rod is defined by the mapping
\[
   [a,b]\times \R \ni (s,t) \mapsto \left(\br(s,t),\,\bd_1(s,t),\,\bd_2(s,t),\bd_3(s,t)\right)\in \E^3\,. \label{rd1d2}
\]

Furthermore, the governing system~\eqref{kinematic1}--\eqref{dynamic2} includes additional vector-valued dependent variables: {\em linear strain} $\bnu$ of the rod and the {\em velocity} $\bv$ of  the material cross-section:
\[
 \bnu:=\partial_s\br=\sum_{k=1}^3\nu_k\bd_k\,,\quad \bv:=\partial_t\br=\sum_{k=1}^3v_k\bd_k\,.
\]

\contourlength{1.6pt}
\newcommand{\renderScale}{0.12}
\newcommand{\outlineScale}{1.25*\renderScale}
\newcommand{\pxToMetric}{0.0352777778}
\newcommand{\horizontalCrop}{0.3}
\newcommand{\imageWidth}{1920}
\newcommand{\imageHeight}{1080}

\begin{figure}
\centering
\begin{tikzpicture}
\begin{scope}[scale=10*\pxToMetric*\renderScale]
\coordinate (d1) at (609.500mm, 393.654mm);
\coordinate (d2) at (451.588mm, 30.955mm);
\coordinate (d3) at (839.662mm, 92.118mm);
\coordinate (o1) at (631.054mm, 174.816mm);
\coordinate (ax) at (887.632mm, 543.492mm);
\coordinate (ay) at (1369.466mm, 541.844mm);
\coordinate (az) at (1126.066mm, 941.984mm);
\coordinate (o2) at (1118.542mm, 639.549mm);
\coordinate (r) at (1495.859mm, 221.154mm);
\coordinate (s0) at (608.083mm, 533.215mm);
\coordinate (sl) at (1832.694mm, 552.899mm);
\end{scope}

\tikzset{
  double -latex/.style args={#1 colored by #2 and #3}{
    -latex,line width=#1,#2,
    postaction={draw,-latex,#3,line width=(#1)/3,shorten <=(#1)/4,shorten >=4.5*(#1)/3},
  },
  double round cap-latex/.style args={#1 colored by #2 and #3}{
    round cap-latex,line width=#1,#2,
    postaction={draw,round cap-latex,#3,line width=(#1)/3,shorten <=(#1)/4,shorten >=4.5*(#1)/3},
  },
  double round cap-stealth/.style args={#1 colored by #2 and #3}{
    round cap-stealth,line width=#1,#2,
    postaction={round cap-stealth,draw,,#3,line width=(#1)/3,shorten <=(#1)/3,shorten >=2*(#1)/3},
  },
  double -stealth/.style args={#1 colored by #2 and #3}{
    -stealth,line width=#1,#2,
    postaction={-stealth,draw,,#3,line width=(#1)/3,shorten <=(#1)/3,shorten >=2*(#1)/3},
  },
  double -triangle 45/.style args={#1 colored by #2 and #3}{
    -triangle 45,line width=#1,#2,
    postaction={-triangle 45,draw,,#3,line width=(#1)/3,shorten <=(#1)/3,shorten >=3.5*(#1)/2},
  },
}

\path[use as bounding box] (\horizontalCrop, 0) rectangle (\renderScale*\pxToMetric*\imageWidth - \horizontalCrop, \renderScale*\pxToMetric*\imageHeight);

\node[scale=\outlineScale, anchor=south west, inner sep=0pt] at (0, 0) {\includegraphics{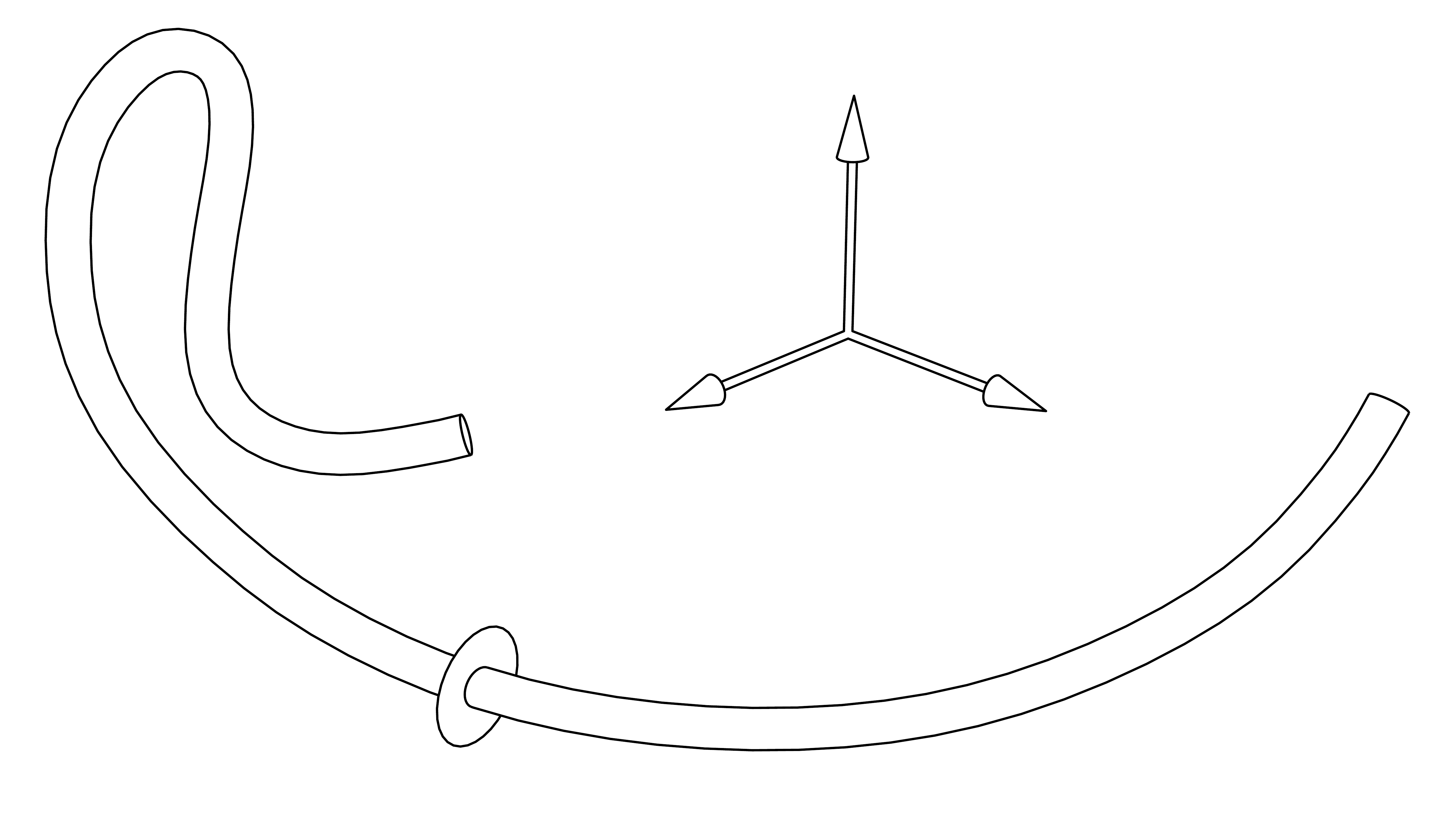}};

\draw[double -latex=2pt colored by black and white] (o2) -- (r) node[black, pos=0.65, right, inner sep=5pt, xshift=0pt, yshift=0pt] {\contour{white}{$\boldsymbol{r}(s, t)$}};

\draw[white, thick, fill=black] (o2) circle (0.1);

\draw[double -latex=2pt colored by black and white] (o1) -- (d1) node[black, anchor=west, inner sep=2pt, xshift=0pt, yshift=0pt] {\contour{white}{$\boldsymbol{d_1}$}};
\draw[double -latex=2pt colored by black and white] (o1) -- (d2) node[black, anchor=south east, inner sep=0pt, xshift=0pt, yshift=0pt] {\contour{white}{$\boldsymbol{d_2}$}};
\draw[double -latex=2pt colored by black and white] (o1) -- (d3) node[black, anchor=west, inner sep=0pt, xshift=0pt, yshift=0pt] {\contour{white}{$\boldsymbol{d_3}$}};

\node [black, anchor=north east, inner sep=2pt, xshift=0pt, yshift=0pt] at (ax) {\contour{white}{$\boldsymbol{e}_1$}};
\node [black, anchor=north west, inner sep=2pt, xshift=0pt, yshift=0pt] at (ay) {\contour{white}{$\boldsymbol{e}_2$}};
\node [black, anchor=south, inner sep=4pt, xshift=0pt, yshift=0pt] at (az) {\contour{white}{$\boldsymbol{e}_3$}};

\node [black, anchor=south, inner sep=3pt, xshift=0pt, yshift=0pt] at (s0) {\contour{white}{$s = a$}};
\node [black, anchor=south, inner sep=3pt, xshift=0pt, yshift=0pt] at (sl) {\contour{white}{$s = b$}};

\end{tikzpicture}
\caption{The vector set $\{\vec{d}_1,\vec{d}_2,\vec{d}_3\}$ forms a right-handed orthonormal basis. The directors $\vec{d}_1$ and $\vec{d}_2$ span the local material cross-section, whereas $\vec{d}_3$ is perpendicular to the cross-section. Note that in the presence of shear deformations  $\vec{d}_3$ is unequal to the tangent $\partial_s\vec{r}$ of the centerline of the rod.}
\label{Fig1}
\end{figure}
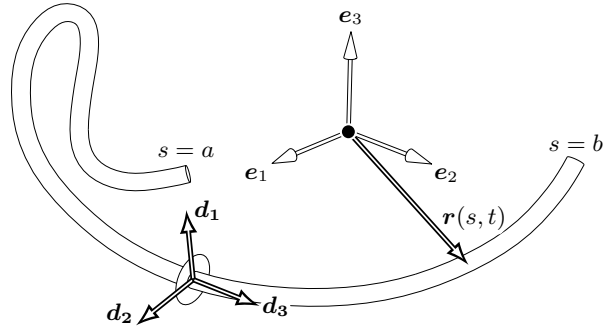

The components of the {\em strain variables} $\bk$ and $\bnu$ describe the deformation of the rod: the flexure with respect to the two major axes of the cross section $(\kappa_1,\kappa_2)$, torsion $(\kappa_3)$, shear $(\nu_1,\nu_2)$, and extension $(\nu_3)$.

The {\em kinematic part} of the governing Cosserat system consists of equations~\refeq{kinematic1}--\refeq{kinematic2} ((9a)--(9b) in~\cite{Antman:1995}, Ch.~8).
The remaining equations~\refeq{dynamic1}--\refeq{dynamic2} ((9c)--(9d) in~\cite{Antman:1995}, Ch.~8) make up the {\em dynamical part} of the governing equations. For a rod density $\rho(s)$ and cross-section $A(s)$, these equations follow from Newton's laws of motion:
\begin{eqnarray*}
\begin{array}{l}
\rho(s) A(s)\partial_t\bv=\partial_s\bn(s,t)+\bF(s,t)\,,\\[0.2cm]
\partial_t\bh(s,t)=\partial_s\bm(s,t)+\bnu(s,t)\times \bn(s,t)+\bl(s,t)\,,
\end{array}
\label{nl}
\end{eqnarray*}
where $\bm(s,t)=\sum_{k=1}^3m_k(s,t)\,\bd_k(s,t)$ are the {\em contact torques},  $\bn(s,t)=\sum_{k=1}^3n_k(s,t)\,\bd_k(s,t)$ are the {\em contact forces}, $\bh(s,t)=\sum_{k=1}^3h_k(s,t)\,\bd_k(s,t)$ are the {\em angular momenta}, and $\bF(s,t)$ and $\bl(s,t)$ are the  {\em external forces} and {\em torque densities}.

The contact torques $\bm(s,t)$ and contact forces $\bn(s,t)$ corresponding to the {\em internal stresses}, are related to the extension and shear strains $\bnu(s,t)$ as well as to the flexure and torsion strains $\bk(s,t)$ by the
{\em constitutive relations}
\begin{equation*} \label{cons_rel}
 \bm(s,t)=\hat{\bm}\left(\bk(s,t),\bnu(s,t),s\right)\,,\quad \bn(s,t)=\hat{\bn}\left(\bk(s,t),\bnu(s,t),s\right)\,.
\end{equation*}
Under certain reasonable assumptions (cf.~\cite{Antman:1995,CaoTucker:2008,MLGSW:2014}) on the structure of the right-hand sides in~\eqref{nl}, they take the form~\eqref{dynamic1}--\eqref{dynamic2} in which $\bJ$ is the inertia tensor of the cross-section per unit length. Unlike the kinematic part, the dynamical part contains parameters characterizing the rod under consideration: $\rho,A$ and $\bJ$ together with the external force $\bF$ and torque $\bl$, whereas the kinematic part is parameter free.

\noindent
In our previous papers~\cite{MLGSW:2014,MLGHRKW:2016}, by treating the kinematic Cosserat equations~\refeq{kinematic1}--\refeq{kinematic2} with computer algebra aided methods of the modern Lie symmetry analysis~(cf.~\cite{Ibragimov:09,Olver:93}) and the theory of completion of partial differential systems to involution (cf.~\cite{Seiler:10}),  we constructed the following closed form of an analytical solution to the kinematic part and proved its generality:
\begin{subequations}
\begin{align}
& \vec{\omega} ={\vec{p}_t}+\frac{p-\sin(p)}{p^3}\,\left(\vec{p}\,\left(\vec{p}\cdot {\vec{p}_t}\right)-p^2\,{\vec{p}_t}\right)-\frac{1-\cos(p)}{p^2}\,\vec{p}\times {\vec{p}_t}\,,\label{omega}\\
& \vec{\kappa} = {\vec{p}_s}
  +\frac{p-\sin(p)}{p^3}\,\left(\vec{p}\,\left(\vec{p}\cdot{\vec{p}_s}\right)-p^2\,{\vec{p}_s}\right)-\frac{1-\cos(p)}{p^2}\,\vec{p}\times {\vec{p}_s}\,,\label{kappa}\\
& \bnu=\bq\times\bk-\bq_s\,,\label{nu}\\[0.25cm]
& \bv=\bq\times\bw-\bq_t\,, \nonumber
\end{align}
\end{subequations}
where $\bp(s,t)$ and $\bq(s,t)$ are arbitrary analytic vector functions, and
\[
p=\sqrt{p_1^2+p_2^2+p_3^2}\,.
\]
For the efficient numerical solving of the dynamical Cosserat equations~\eqref{dynamic1}--\eqref{dynamic2}, we use the following fact: the vector equation~\eqref{omega} uniquely defines $\vec{p}_t$ in terms of $\vec{p}$ and $\vec{\omega}$. We formulate this fact as the following statement.
\begin{proposition}\label{pro:1}
The temporal derivative $\vec{p}_t$ of the vector function $\vec{p}$, as a solution of~\eqref{omega}, reads
\begin{equation}\label{pt-solution}
\vec p_t=\dfrac{\vec p\cdot\vec{\omega}}{p^2}\;\vec p+\frac 1 2\;\vec p\times\vec{\omega}
-\frac p 2 \cot\left(\frac{p}{2}\right)\cdot \frac{\vec p\times(\vec p\times\vec{\omega})}{p^2}\,.
\end{equation}
\end{proposition}

\begin{proof}
The vector function $\vec{p}_t$ occurs linearly in~\eqref{omega}. In the component form it is a linear system of three equations in three unknowns $(p_1)_t,(p_2)_t,(p_3)_t$ whose matrix has a non-singular determinant (cf. formula~(11) in~\cite{MLGHRKW:2016})
\begin{equation}
2\,\frac{\cos(p)-1}{p^2} \label{Jacobian}\,.
\end{equation}

The vector form of $\vec{p}_t$ as a solution to equality~\eqref{omega} is given by~\eqref{pt-solution}. This can be verified either by hand computation or by using the routines of the Maple package {\em{VectorCalculus}} after the substitution of~\eqref{pt-solution} into the right-hand side of~\eqref{omega} and simplification of the obtained expression to $\vec{\omega}$.\\\phantom{}\hfill $\Box$
\end{proof}
Instead of system~\eqref{kinematic1}--\eqref{dynamic2} for unknowns $(\vec{\omega}, \vec{\kappa}, \bnu, \bv)$, we are going to solve the equivalent system
\begin{subequations}
 \begin{align}
    & \vec p_t=\dfrac{\vec p\cdot\vec{\omega}}{p^2}\;\vec p+\frac 1 2\;\vec p\times\vec{\omega} -\frac p 2 \cot\left(\frac{p}{2}\right)\cdot \frac{\vec p\times(\vec p\times\vec{\omega})}{p^2}\,, \label{eqp:1}\\[0.1cm]
    & \bq_t=\bq\times\bw-\bv\,,\label{eqp:2}\\[0.15cm]
    & \rho\bJ\cdot \bw_t=\hat{\bm}_s+\bk\times\hat{\bm}+\bnu\times
        \hat{\bn}-\bw\times(\rho\bJ\cdot\bw)+\bl\,,\label{eqp:3}\\[0.15cm]
    & \rho A\bv_t=\hat{\bn}_s+\bk\times \hat{\bn}-\bw\times (\rho A\bv)+\bF \label{eqp:4}
 \end{align}
\end{subequations}
for unknown vector functions $(\vec{p}, \vec{q}, \vec{\omega}, \bv)$, where $\bk$ and $\bnu$ are given by \eqref{kappa}--\eqref{nu}.

\section{Symbolic-Numeric Integration Method}
\label{sec:3}

\subsection{Naive Approach: Explicit Numerical Solving}

Suppose we know the values of the vector-functions $\bw$ and $\vec{p}$ on a time layer $t$. Then $\vec{p}_t$ in~\eqref{eqp:1} can be approximated by the forward Euler difference
\[
\vec{p}_t \rightarrow \frac{\vec{p}(s,t+\triangle t) - \vec{p}(s,t)}{\triangle t}\,.
\]
However, since equation~\eqref{eqp:1} has a singularity at $p = 2 \pi$ related with vanishing~\eqref{Jacobian}, there is a restriction to the time step $\Delta t$ caused by the condition
\begin{equation}\label{p-interval}
p(s,t+ \triangle t) \in (0, 2 \pi)
\end{equation}
to be held for all values of $s$. This is a severe problem for the numerical solving of the governing Cosserat system, because in the course of solving, one must control the time step at every value of $s$ to keep the values of $p(s,t)$ within the interval indicated in~\eqref{p-interval}. This problem is resolved in the symbolic-numeric integration method described in the next subsection.

\subsection{Advanced Approach based on Exponential Integration}

To avoid the problem of controlling the condition~\eqref{p-interval} we use the differential  equation~\eqref{eqp:1} for $\vec{p}(t)$ and rewrite it in terms of $p$ and the unit vector $\vec{e}$
where $\vec{p} = p \vec{e}$\,. It leads to the following differential system:
\begin{subequations}
 \begin{align}
    & p_t=\vec{e}\cdot \vec{\omega}\,,\nonumber\\[0.1cm]
    & 2\,\vec{e}_t=\vec{e}\times \vec{\omega}-\cot\left(\frac{p}{2}\right)\vec{e}\times(\vec{e}\times \vec{\omega})\,.\nonumber
 \end{align}
\end{subequations}
Now assume that the vector $\vec{\omega}$ is independent of $t$ on the time interval $\Delta t$ and choose the Cartesian coordinate system $\vec{e}_1, \vec{e}_2, \vec{e}_3$ such that $\vec{e}_3 || \vec{\omega}$:
\[
\vec{e} = A_1\, \vec{e}_1 + A_2\, \vec{e}_2 + A_3\, \vec{e}_3\,,\quad \vec{\omega} = \omega\, \vec{e}_3\,,\quad \omega:=\sqrt{\omega_1^2+\omega_2^2+\omega_3^2}\,.
\]
Then, we obtain the following system of four first-order differential equations:
\begin{subequations}
 \begin{align}
& 2\, (A_1)_{t} = A_2\, \omega - \cot\left(\frac{p}{2}\right) A_1 A_3\, \omega\,, \label{eq:1}\\[0.15cm]
& 2\, (A_2)_{t} = -A_1\, \omega - \cot\left({\frac{p}{2}}\right) A_2 A_3\, \omega\,,\label{eq:2}\\[0.15cm]
& 2\, (A_3)_{t} = -\cot\left(\frac{p}{2}\right) (A_3^2 - 1)\, \omega\,,\label{eq:3}\\[0.15cm]
& p_t = A_3\, \omega\,. \label{eq:4}
\end{align}
\end{subequations}
From the equations \eqref{eq:3}--\eqref{eq:4} it follows
\[
  \frac{2\,A_3(A_3)_t}{A_3^2-1}=-\cot\left(\frac{p}{2}\right)p_t\,,
\]
and, hence,
\begin{equation}\label{eq:p}
   (1-A_3^2)\,\sin^2\left(\frac{p}{2}\right)=C\,,\quad C_t=0\,.
\end{equation}
Equation~\eqref{eq:p} immediately implies the following statement providing fulfillment of~\eqref{p-interval}.

\begin{proposition}
\label{pro:2}
If $C \ne 0$, then $p(t)\in (0,2\pi)$ for all $t\geq t_0$ if $p(t_0)\in (0,2\pi)$.
\end{proposition}
If one substitutes $A_3=p_t/\omega$ from~\eqref{eq:3} and replaces $p$ with $q:=\cos\left(\frac{p}{2}\right)$, then \eqref{eq:p} takes the form
\[
 4\,q_t^2=\omega^2\left(1-q^2 - C\right)\,.
\]

This equation is easily solvable by the Maple routine {\em dsolve} which outputs four solutions. These solutions can be unified into the general solution
\begin{equation*}\label{q-solution}
  q= \sqrt{1 - C} \sin\left(\frac{1}{2} \omega (C_1 - t) \right)\,,\quad (C_1)_t=0\,.
\end{equation*}
Then, the whole system~\eqref{eq:1}--\eqref{eq:4} admits the following general analytical solution
\begin{subequations}
 \begin{align}
& A_1(s,t) = - \frac{\sqrt{C} \cdot sin(\frac{1}{2}\omega(C_2 - t))}{\sqrt{\omega^2 \cos^2(\frac{1}{2}\omega (C_1 - t)) + C \sin^2(\frac{1}{2}\omega (C_1 - t))}}\,,\label{ExpSol:1}\\[0.15cm]
& A_2(s,t) = \frac{\sqrt{C} \cdot cos(\frac{1}{2}\omega(C_2 - t))}{\sqrt{\omega^2 \cos^2(\frac{1}{2}\omega (C_1 - t)) + C \sin^2(\frac{1}{2}\omega (C_1 - t))}}\,,\label{ExpSol:2}\\[0.15cm]
& A_3(s,t) = \frac{\sqrt{\omega^2 - C} \cdot cos(\frac{1}{2}\omega(C_1 - t))}{\sqrt{\omega^2 \cos^2(\frac{1}{2}\omega (C_1 - t)) + C \sin^2(\frac{1}{2}\omega (C_1 - t))}}\,,\label{ExpSol:3}\\[0.15cm]
& p(s,t) = 2\arccos \left(\frac{\sqrt{\omega^2 - C}\sin(\frac{1}{2} \omega (C_1 - t))}{\omega} \right)\,,\label{ExpSol:4}
\end{align}
\end{subequations}
where $C, C_1, C_2$ are functions of $s$. These functions are determined by the following initial data:
\begin{eqnarray}
&& C(s) := \omega^2 \left(1 - A_3^2(s,t_0)\right)\sin^2\left(\frac{p(s,t_0)}{2}\right)\,,\nonumber \\
&& C_1(s) := t_0 + \frac{A_3(s,t_0) |\sin(p(s,t_0))|}{\sqrt{\omega^2 - C(s)}}\,, \nonumber\\
&& C_2(s) := t_0 + \frac{2}{\omega} \arctan \left( \frac{A_1(s,t_0)}{A_2(s,t_0)} \right)\,.\nonumber
\end{eqnarray}

\begin{proposition}
\label{pro:3}
$C(s) \equiv 0$ if and only if $A_3(s,t) = \pm 1$ what corresponds to a degenerated solution\\[-0.3cm]
$$
A_1(s,t) = 0\,,\quad A_2(s,t) = 0\,,\quad p(s,t) = p(s,t_0) \pm \omega t\,.
$$
Computationally, this solution is not of interest, since it is unstable: a small (e.g. numerical) deviation of $A_3(s,t_0)$ from $\pm 1$ converts the solution into a generic one.
\end{proposition}

\begin{proof} The solution\\[-0.3cm]
$$
p(s,t) = p(s,t_0) \pm \omega t, A_3(s,t) = \pm 1, A_2(s,t) = A_1(s,t) = 0
$$
is singular. Unlike the generic solution, the value of $|p(t)|$ in this solution may increase indefinitely.
However, it is unstable, since any small perturbation $\epsilon > 0$ to the initial value $A_3(s,t_0) = \pm \left(1 - \epsilon(s)\right)$ leads to the generic case when $p(t)$ remains bounded.\\\phantom{}\hfill $\Box$
\end{proof}

Our symbolic-numeric approach to the derivation of equations \eqref{eq:1}--\eqref{eq:4} and the construction of their explicit analytical solution~\eqref{ExpSol:1}--\eqref{ExpSol:4} is in accord with the general principles of exponential integration (see e.g.~\cite{Hochbruck:2010}). The basic idea behind exponential integration is the identification of a prototypical differential system which has the stiffness properties similar to those in the original equation and which admits explicit solving.

In our case, the stiffness properties of the differential system~\eqref{eq:1}--\eqref{eq:4} are similar to those in system~\eqref{eqp:1}--\eqref{eqp:4}. In doing so, the last system belongs to the second class of stiff problems (cf.~\cite{Hochbruck:2010}, p.~210) whose stiffness is caused by the highly oscillatory behavior of their solutions. For such problems both explicit and implicit Euler schemes fail to provide the required stability unless the step size is strongly reduced to provide the resolution of all the oscillations in the solution. Thereby, the standard numerical treatment of the equations \eqref{kinematic1}--\eqref{dynamic2} and hence equations~\eqref{eqp:1}--\eqref{eqp:4} is computationally inefficient. Just by this reason, special numerical solvers have been designed (cf.~\cite{LangLinnArnold:2011,SobottkaLayWeber:2008}) for the equations \eqref{kinematic1}--\eqref{dynamic2}.

\reffig{fig:2} illustrates the stiffness of the differential system~\eqref{eqp:1}--\eqref{eqp:4}. The behavior is shown, at $\omega=1$, of the functions $p(s_0, t)$ and $A_3(s_0, t)$ as solutions of \eqref{eq:3}--\eqref{eq:4} for the initial conditions $p(s_0,0)=1$ and $A_3(s_0,0)=0.99$. As illustrated, the solution oscillates and changes drastically over time. A numerical reconstruction of such a behavior is possible only for very small step sizes of difference approximations.

\begin{figure}[!tbp]
  \centering
  \begin{minipage}[b]{0.475\textwidth}
    \includegraphics[width=\textwidth]{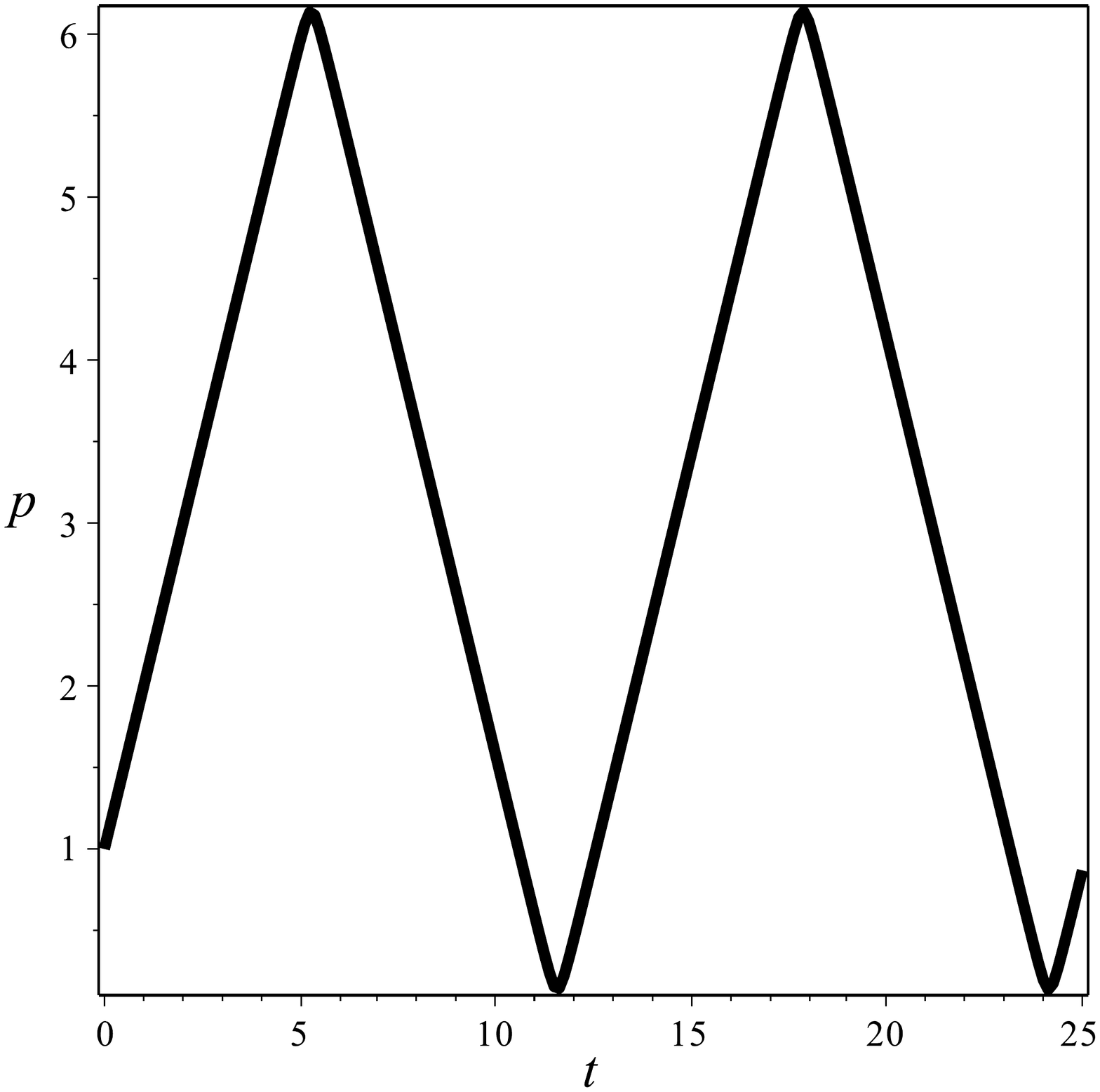}
  \end{minipage}
  \hfill
  \begin{minipage}[b]{0.475\textwidth}
    \includegraphics[width=\textwidth]{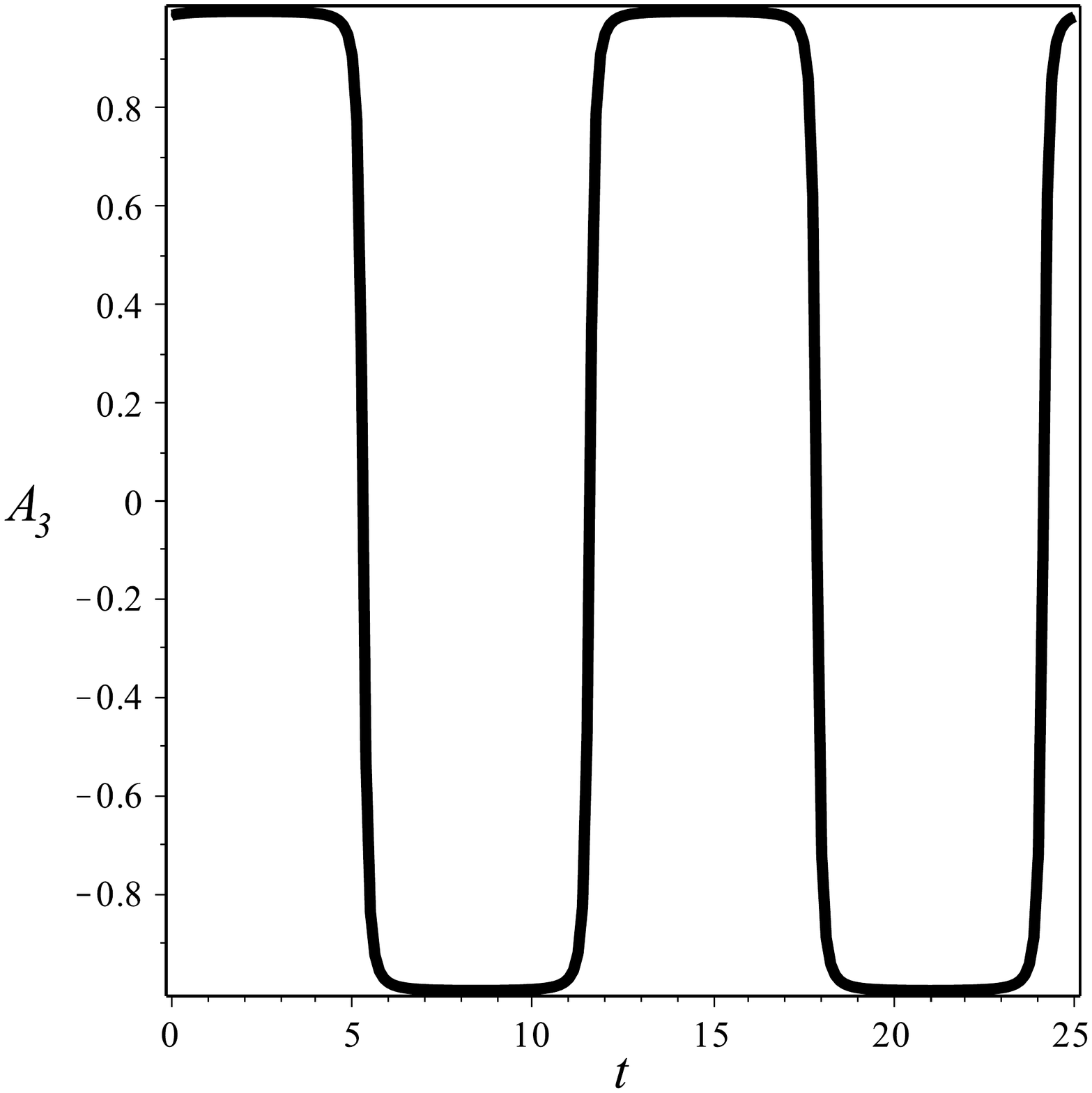}
  \end{minipage}
  \vspace{-1cm}
  \caption{Illustration of the temporal evolution of $p(s_0,t)$ (left) and  $A_3(s_0,t)$ (right).}
      \label{fig:2}
\end{figure}

\section{Numerical Comparison with the Generalized $\alpha$-Method}
\label{sec:4}

In 1993, Chung and Hulbert (cf.~\cite{Chung:1993}) presented the generalized $\alpha$-method as a new integration algorithm for problems from structural mechanics. It is characterized primarily by a controllable numerical dissipation of high-frequency components in the numerical solution. These occur, for example in the context of finite element-based simulations, when the high-frequency states are too roughly resolved. Such methods usually improve the convergence behavior of iterative solving strategies for nonlinear problems.

The generalized $\alpha$-method is well-established in the field of structural mechanics and has the major advantage of unconditional stability as well as user controllable numerical damping. The idea of the introduction of a controllable numerical damping in the integration process is not new and found, among other things, realization in $\alpha$-HHT (cf.~\cite{Hilbert:1977}) or in the WBC-$\alpha$-method (cf.~\cite{Wood:1981}). The aim of the development of such methods is to maximize the attenuation of high-frequency components while preserving the important low-frequency components. Using the generalized $\alpha$-method, this ratio is optimal, i.e.~for a given attenuation of the high-frequency components, the attenuation of the low-frequency components is minimized. A brief specification of this method for general problems in structural mechanics in given in Appendix \refapx{sec:appendix}.

Following \cite{SobottkaLayWeber:2008}, using a state vector\\[-0.3cm]
$$
\vec{x}(s,t) = \left(\vec{v}(s,t),\vec{\omega}(s,t),\vec{\kappa}(s,t),\vec{n}(s,t)\right)^{\mathsf{T}}
$$
describing the rod, we can rewrite its equations of motion \refeq{kinematic1}--\refeq{dynamic2} in terms of a system
\begin{equation}
\label{cos:alpha}
\hat{\mat{M}}\partial_t\vec{x}(s,t)+\hat{\mat{K}}\partial_s\vec{x}(s,t)+\vec{\Lambda}(s,t)=\vec{0}\,.
\end{equation}
Here\\[-0.5cm] $$
\hat{\mat{M}}=\mathsf{diag}(\rho A,\rho A,\rho A,\rho I_1,\rho I_2, \rho I_3,1,1,1,0,0,0)
$$\\[-0.3cm]
is the mass matrix and $\hat{\mat{K}}=-\mathsf{adiag}(\mat{1},\mat{1},\mat{K},\mat{1})$ is the stiffness matrix with $\mat{K}=\mathsf{diag}(EI_1,EI_2,G\mu)$, in which the bending stiffness into the direction of the principal components of the cross section $A$ is denoted by $EI_{1,2}$, and the torsional stiffness by $G\mu$. As above, the rod's density is given by $\rho$, the Young's modulus by $E$ and the shear modulus by $G$.
The nonlinear terms are included in the nonlinearity $\vec{\Lambda}$. We do not explicitly write out the resulting equations here for brevity and refer to \cite{SobottkaLayWeber:2008} for the explicit form of \eqref{cos:alpha}. In the generalized $\alpha$-method, the update schemes of positions and velocities at point $i$ in time correspond to those of the classical Newmark integrator. Its accurate and efficient application in the context of the simulation of elastic rods was demonstrated in \cite{SobottkaLayWeber:2008}.

According to \cite{SobottkaLayWeber:2008}, we consider four different test cases: (i) a sinus-like shaped rod which is released under gravity from a horizontal position (no damping); (ii) a highly damped helical rod subject to a time-varying end point load (the damping is obtained by setting the integration parameters (see Appendix \refapx{sec:appendix}) of the generalized $\alpha$-method to $\alpha:=\alpha_m=\alpha_f=0.4$, $\beta=0$, and $\gamma=1.0$); (iii) a straight rod (45 cm) subject to a time-varying torque; (iv) a helical rod with low damping that is excited by a force parallel to the axis of the helix and released after $0.1\,\text{s}$ showing the typical oscillating behavior of a steel-like coil spring.

We simulate these scenarios using the generalized $\alpha$-method and the symbolic-numeric integration scheme presented in this contribution. All fibers are discretized using $100$ individual segments. Using the symbolic-numeric method, damping is incorporated using the linear Rayleigh damping as described in \cite{MMS:2015}. We enforce a maximally tolerated relative $L^2$-error of 1\% in the position and velocity space in order to ensure sufficient accuracy and measure the required computation time on a machine with an Intel(R) Xeon E5 with 3.5 GHz and 32 GB DDR-RAM without parallelization. For all test cases, we obtain significant speedups of the presented symbolic-numeric method (``snm'') compared to the generalized $\alpha$-method (``$\alpha$''), in particular:
\begin{enumerate}[(i)]
\item speedup of a factor more than 20 ($\alpha$: $4.1\,\text{s}$; snm: $0.2\,\text{s}$);
\item speedup of over $21\times$ ($\alpha$: $4.3\,\text{s}$; snm: $0.2\,\text{s}$);
\item speedup of approx.~$19\times$ ($\alpha$: $3.8\,\text{s}$; snm: $0.2\,\text{s}$);
\item speedup of approx.~$34\times$ ($\alpha$: $6.8\,\text{s}$; snm: $0.2\,\text{s}$).
\end{enumerate}
Please note that the computation time of $0.2\,\text{s}$ of the presented symbolic-numeric method is constant for all test cases (with identical duration of $8\,\text{s}$).

\section{Conclusison}
\label{sec:5}

Based on the closed form solution to the kinematic part~\eqref{kinematic1}--\eqref{kinematic2} of the governing Cosserat system~\eqref{kinematic1}--\eqref{dynamic2} and with assistance of Maple, we developed a new symbolic-numeric method for their integration. Our computational experiments demonstrate the superiority of the new method over the generalized $\alpha$-method for the accurate and efficient integration of Cosserat rods. Its application prevents from numerical instabilities and allows for highly accurate and efficient simulations. This clearly shows the usefulness of the constructed analytical solution to the kinematic equations and should enable more complex and realistic Cosserat rod-based scenarios to be explored in scientific computing without compromising efficiency.

\section*{Acknowledgements}
The authors appreciate the insightful comments of the anonymous referees. This work has been partially supported by the King Abdullah University of Science and Technology (KAUST baseline funding), the Max Planck Center for Visual Computing and Communication (MPC-VCC) funded by Stanford University and the Federal Ministry of Education and Research of the Federal Republic of Germany (BMBF grants FKZ-01IMC01 and FKZ-01IM10001), the Russian Foundation for Basic Research (grant 16-01-00080) and the Ministry of Education and Science of the Russian Federation (agreement 02.a03.21.0008).



\appendix
\section{Generalized $\alpha$-Method}
\label{sec:appendix}

In this appendix, we briefly explain the application of the generalized $\alpha$-method for the common case of a system described by the standard equation from structural mechanics,
\begin{equation} \label{structural_mechanics}
\vec{M}\ddot{\vec{x}}+\vec{D}\dot{\vec{x}}+\vec{K}\vec{x}+\vec{\Lambda}(t)=\vec{0},
\end{equation}
in which $\vec{M}$,$\vec{D}$, and $\vec{K}$ denote the mass, damping, and stiffness matrices. The time-dependent displacement vector is given by $\vec{x}(t)$ and its first- and second-order temporal derivatives describe velocity and acceleration. The vector $\vec{\Lambda}(t)$ describes external forces acting on the system at time $t$. We are searching for functions $\vec{x}(t)$, $\vec{\upsilon}(t)=\dot{\vec{x}}(t)$, and $\vec{a}(t)=\ddot{\vec{x}}(t)$ satisfying \eqref{structural_mechanics} for all $t$ with initial conditions $\vec{x}(t_0)=\vec{x}_0$ and $\vec{\upsilon}(t_0)=\vec{\upsilon}_0$.

For the employment of the generalized $\alpha$-method, we can write the integration scheme with respect to \eqref{structural_mechanics} as follows:
$$
\vec{M}\vec{a}_{1-\alpha_m}+\vec{D}\vec{x}_{1-\alpha_f}+\vec{K}\vec{x}_{1-\alpha_f}+\vec{\Lambda}(t_{1-\alpha_f})=\vec{0},
$$
with the substitution rule $(\cdot)_{1-\alpha}:=(1-\alpha)(\cdot)_i+\alpha(\cdot)_{i-1}$ and the approximations
\begin{eqnarray*}
\vec{x}_i=\vec{x}_{i-1}+\Delta t\vec{v}_{i-1}+\Delta t^2\left(\left(\frac{1}{2}-\alpha\right)\vec{a}_{i-1}+\beta\vec{a}_i\right)\,,\\
\vec{\upsilon}_i=\vec{\upsilon}_{i-1}+\Delta t\left(\left(1-\gamma\right)\vec{a}_{i-1}+\gamma\vec{a}_i\right)\,.
\end{eqnarray*}
The parameters $\alpha_m$, $\alpha_f$, $\gamma$, and $\beta$ are integration coefficients.

\end{document}